\documentclass[onefignum,onetabnum]{siamonline190516}

\usepackage{amsmath}
\usepackage{amssymb}
\usepackage{enumerate}
\usepackage{mathrsfs}
\usepackage{color}
\usepackage{tikz}

\usepackage{lipsum}
\usepackage{amsfonts}
\usepackage{graphicx}
\usepackage{epstopdf}
\usepackage{algorithmic}
\ifpdf
  \DeclareGraphicsExtensions{.eps,.pdf,.png,.jpg}
\else
  \DeclareGraphicsExtensions{.eps}
\fi

\usepackage{enumitem}
\setlist[enumerate]{leftmargin=.5in}
\setlist[itemize]{leftmargin=.5in}

\newsiamremark{remark}{Remark}
\newsiamremark{hypothesis}{Hypothesis}
\crefname{hypothesis}{Hypothesis}{Hypotheses}
\newsiamthm{claim}{Claim}

\headers{Nonnegative Polynomials and Circuit Polynomials}{Jie Wang}

\title{Nonnegative Polynomials and Circuit Polynomials\thanks{Submitted to the editors DATE.
\funding{This work was supported by NSFC under grants 61732001 and 61532019.}}}

\author{Jie Wang\thanks{Laboratory of Analysis and Architecture of Systems, French National Centre for Scientific Research
  (\email{jwang@laas.fr}, \url{https://wangjie212.github.io/jiewang/}).}}

\usepackage{amsopn}

\newsiamthm{example}{Example}

\def\R{{\mathbb{R}}}
\def\C{{\mathbb{C}}}
\def\N{{\mathbb{N}}}

\def\x{{\mathbf{x}}}
\def\z{{\mathbf{z}}}

\def\bb{{\mathbf{b}}}
\def\a{{\boldsymbol{\alpha}}}
\def\dd{{\boldsymbol{\delta}}}
\def\b{{\boldsymbol{\beta}}}
\def\bv{{\boldsymbol{v}}}
\def\bu{{\boldsymbol{u}}}
\def\A{{\mathscr{A}}}
\def\I{{\mathscr{I}}}

\def\C{{\mathscr{C}}}

\def\supp{\hbox{\rm{supp}}}
\def\PSD{\hbox{\rm{PSD}}}
\def\SONC{\hbox{\rm{SONC}}}

\def\int{\hbox{\rm{int}}}
\def\New{\hbox{\rm{New}}}
\def\Conv{\hbox{\rm{conv}}}
\def\rank{\hbox{\rm{rank}}}
\def\Cone{\hbox{\rm{cone}}}

\newcommand{\revision}[1]{{{\color{black}#1}}}

\begin{document}

\maketitle

% REQUIRED
\begin{abstract}
  The concept of sums of nonnegative circuit polynomials (SONC) was recently introduced as a new certificate of nonnegativity especially for sparse polynomials. In this paper, we explore the relationship between nonnegative polynomials and SONC polynomials. As a first result, we provide sufficient conditions for nonnegative polynomials with general Newton polytopes to be SONC polynomials, which generalizes the previous result on nonnegative polynomials with simplex Newton polytopes. Secondly, we prove that every SONC polynomial admits a SONC decomposition without cancellation. In other words, SONC decompositions can exactly preserve the sparsity of nonnegative polynomials, which is dramatically different from the classical sum of squares (SOS) decompositions and is a key property to design efficient algorithms for sparse polynomial optimization based on SONC decompositions.
\end{abstract}

% REQUIRED

% REQUIRED

\section{Introduction}
A real polynomial $f\in\R[\x]=\R[x_1,\ldots,x_n]$ is called a {\em nonnegative polynomial} if its evaluation on every real point is nonnegative. All nonnegative polynomials form a convex cone, denoted by PSD. Certifying nonnegativity of multivariate polynomials is a central problem of real algebraic geometry and also has a deep connection with polynomial optimization. A classical approach for handling this problem is using sum of squares (SOS) decompositions. From the perspective of computation, checking whether a polynomial is a sum of squares boils down to a semidefinite program (SDP) involving a positive semidefinite matrix of size $\binom{n+d}{d}$, where $n$ is the number of variables and $2d$ is the degree of the polynomial \cite{pa}. Hence, the size of the corresponding SDP problem grows combinatorially with $n,d$, which greatly limits the scalability of this approach given the current state of SDP solvers. 

To address the issue of scalability, one possibility is to exploit the structure in the polynomial data, such as symmetry \cite{symmetry}, correlative sparsity \cite{waki}, term sparsity \cite{wang1,wang3,wang2}, correlative-term sparsity \cite{wang4}, just to name a few. Another possibility is to rely on other nonnegativity certificates. Such alternative nonnegativity certificates are in general more restrictive but cheaper to implement, e.g. (scaled) diagonally dominant sums of squares \cite{ahmadi2019dsos}. \revision{However, a common drawback shared by these approaches is that their computational complexity heavily depends on the polynomial degree. As an attempt to overcome this}, Iliman and de Wolff proposed the concept of {\em sums of nonnegative circuit polynomials} (SONC) as a new nonnegativity certificate of polynomials \cite{iw}.
A {\em circuit polynomial} is of the form
$\sum_{\a\in\A}c_{\a}\x^{\a}-d\x^{\b}\in\R[\x],$
where $c_{\a}>0$ for all $\a\in\A$, $\A\subseteq(2\N)^n$ comprises the vertices of a simplex, and $\b$ lies in the interior of this simplex. The support of a circuit polynomial is called a {\em circuit}. The study of circuit polynomials dates back to 1980s by Reznick \cite{re1} in the special case of simplicial agiforms. \revision{After over two decades of quiescence}, a nonnegativity condition for circuit polynomials was given by Paneta, Koeppl and Craciu in the study of biochemical reaction networks \cite{pantea}, and the subject was brought back to people's view. A related certificate, {\em sums of AGE polynomials} (SAGE), was also recently proposed by Murray, Chandrasekaran and Wiermann \cite{murray}, where an AGE polynomial is defined by a nonnegative polynomial with at most one term that can take negative values (called a negative term). The set of nonnegative polynomials that admit SONC decompositions forms a convex cone, i.e. the SONC cone, and the set of nonnegative polynomials that admit SAGE decompositions forms a convex cone, i.e. the SAGE cone. SONC has been leveraged to solve sparse polynomial optimization via geometric programming \cite{diw,lw,pap,se} by Dressler et al. or via second order cone programming \cite{magron2020sonc,socp} by the author and Magron. SAGE has been leveraged to solve sparse polynomial/signomial optimization via relative entropy programming \cite{murray2,murray} by Murray et al. From the perspective of theory, it is natural to ask:
\begin{enumerate}
    \item Which types of nonnegative polynomials are SONC polynomials? Can we provide sufficient conditions for a nonnegative polynomial to admit a SONC decomposition in terms of the support?
    \item What is the relationship bewtween the SONC cone and the SAGE cone?
\end{enumerate}

In \cite{iw}, Iliman and de Wolff proved that if the Newton polytope of a polynomial $f$ is a simplex and there exists a point such that all terms of $f$ except for those corresponding to the vertices of the Newton polytope take negative values on this point, then $f$ is nonnegative if and only if $f$ admits a SONC decomposition (see Theorem \ref{nc-thm2}). The first contribution of this paper is that we generalize this conclusion to polynomials with general Newton polytopes. Particularly, we prove that for a polynomial with one negative term, it is nonnegative if and only if it admits a SONC decomposition (Theorem \ref{intro:thm1}). 

\revision{\begin{theorem}\label{intro:thm1}
Let $f=\sum_{i=1}^mc_i\x^{\a_i}-d_{0}\x^{\b}\in\R[\x]$ with $\a_i\in(2\N)^n,c_i>0,i=1,\ldots,m$. Then $f\in\PSD$ if and only if $f\in\SONC$.
\end{theorem}

Note that Theorem \ref{intro:thm1} tells us that any AGE polynomial admits a SONC decomposition.} As an immediate corollary, we obtain that the SAGE cone and the SONC cone are actually identical. \revision{Taking a step further}, we also provide sufficient conditions for nonnegative polynomials with multiple negative terms admitting a SONC decomposition in terms of the combinatorial structure of supports (Theorem \ref{intro:thm2}).

\revision{\begin{theorem}\label{intro:thm2}
Let $f=\sum_{i=1}^mc_i\x^{\a_i}-\sum_{j=1}^ld_j\x^{\b_j}\in\R[\x]$ with $\a_i\in(2\N)^n,c_i>0,i=1,\ldots,m$ and $d_j<0,j=1,\ldots,l$. Under a technical condition on the Newton polytope of $f$ and assuming that all $\b_j$ lie in the same side of every hyperplane determined by points among $\{\a_1,\ldots,\a_m\}$, then $f\in\PSD$ if and only if $f\in\SONC$.
\end{theorem}}

\revision{From the perspective of computation, computing SONC decompositions encounters the potential obstacle of enumerating exponentially many circuits since the number of lattice points contained in the Newton polytope grows exponentially with the number of variables and the polynomial degree.}
In order to develop efficient algorithms for certifying nonnegativity and polynomial optimization based on SONC decompositions, a core issue that must be addressed is: which circuits are really needed when constructing SONC decompositions for a given polynomial? As the second contribution of this paper, we clarify an important fact that every SONC polynomial can decompose into a sum of nonnegative circuit polynomials by merely using the support of the original polynomial. In other words, SONC decompositions can exactly preserve the sparsity of polynomials. \revision{Actually, more is true.} we prove that every SONC polynomial admits a SONC decomposition without cancellation \revision{via a connection with sums of binomial squares (Theorem \ref{intro:thm3}). This is dramatically different from the SOS decompositions of nonnegative polynomials, for which extra support and cancellation are needed in general. 

\revision{\begin{theorem}\label{intro:thm3}
If a polynomial $f\in\SONC$, then $f$ decomposes into a sum of nonnegative circuit polynomials without cancellation.
\end{theorem}}

Theorem \ref{intro:thm3} hence provides a significant step towards bypassing the bottleneck of enumerating all circuits in the computation of SONC decompositions. In fact, this result also implies that the complexity of SONC/SAGE certificates does not depend on the polynomial degree, a sharp contrast with SOS-based certificates.}

% Thanks to this key property, when looking for a SONC decomposition it suffices to consider circuits that are contained in the support of the original polynomial, which significantly reduces the number of candidate circuits.

The rest of this paper is organized as follows. In Section \ref{sec2}, we recall some basic facts on SONC polynomials. After that we consider the problem which types of nonnegative polynomials are SONC polynomials. We deal with the case of nonnegative polynomials with one negative term in Section \ref{sec3} and deal with the case of nonnegative polynomials with multiple negative terms in Section \ref{sec4}. In Section \ref{sec5}, we prove that every SONC polynomial decomposes into a sum of nonnegative circuit polynomials without cancellation. Conclusions and discussions are given is Section \ref{sec6}.

\section{Preliminaries}\label{sec2}
\subsection{Notation and nonnegative polynomials}
Let $\R[\x]=\R[x_1,\ldots,x_n]$ be the ring of real $n$-variate polynomials. Let $\R^*$ be the set of nonzero real numbers, $\R_+$ the set of positive real numbers, and $\R_{\ge0}$ the set of nonnegative real numbers. We use boldface to indicate a (column) vector, e.g., $\a=[\alpha_1,\ldots,\alpha_n]^\intercal$. For a finite set $\A\subseteq\N^n$, we denote by $\Cone(\A)$ the conic hull of $\A$, by $\Conv(\A)$ the convex hull of $\A$, and by $V(\A)$ the vertices of the convex hull of $\A$. We also denote by $V(P)$ the vertex set of a polytope $P$. We consider a polynomial $f\in\R[\x]$ supported on a finite set $\A\subseteq\N^n$, i.e. $f$ is of the form $f(\x)=\sum_{\a\in\A}c_{\a}\x^{\a}$ with $c_{\a}\in\R, \x^{\a}=x_1^{\alpha_1}\cdots x_n^{\alpha_n}$. The {\em support} of $f$ is $\supp(f):=\{\a\in\A\mid c_{\a}\ne0\}$ and the {\em Newton polytope} of $f$ is defined as $\New(f):=\Conv(\supp(f))$. For a polytope $P$, we use $P^{\circ}$ to denote the interior of $P$. For a positive integer $m$, let $[m]:=\{1,\ldots,m\}$.

A polynomial $f\in\R[\x]$ which is nonnegative over $\R^n$ is called a {\em nonnegative polynomial}. The class of nonnegative polynomials is denoted by PSD, which forms a convex cone.

A nonnegative polynomial must satisfy the following necessary conditions.
\begin{proposition}(\cite[Theorem 3.6]{re1})\label{nc-prop2}
Let $\A\subseteq\N^n$ and $f=\sum_{\a\in \A}c_{\a}\x^{\a}\in\R[\x]$ with $\supp(f)=\A$. Then $f$ is nonnegative only if the following hold:
\begin{enumerate}
  \item $V(\A)\subseteq(2\N)^n$;
  \item If $\a\in V(\A)$, then the corresponding coefficient $c_{\a}$ is positive.
\end{enumerate}
\end{proposition}

For the remainder of this paper, we assume for simplicity that the monomial factor of any polynomial $f$ is $1$, that is, if $f=\x^{\a'}(\sum c_{\a}\x^{\a})$ such that $\sum c_{\a}\x^{\a}\in\R[\x]$ and $\a'\in\N^n$, then $\x^{\a'}=1$. Otherwise, we can always factor out the monomial factor.

\subsection{Circuit polynomials}
Following \cite{re1}, a subset $\A\subseteq(2\N)^n$ is called a {\em trellis} if $\A$ comprises the vertices of a simplex.
\begin{definition}
Let $\A$ be a trellis and $f\in\R[\x]$. Then $f$ is called a {\em circuit polynomial} if it is of the form
\begin{equation}\label{nc-eq}
f(\x)=\sum_{\a\in\A}c_{\a}\x^{\a}-d\x^{\b},
\end{equation}
with $c_{\a}\in\R_+$ and $\b\in\Conv(\A)^{\circ}$. The support of a circuit polynomial is called a {\em circuit}. 
\end{definition}

\revision{\begin{example}
The Motzkin polynomial $f=1+x^4y^2+x^2y^4-3x^2y^2$ is a nonnegative circuit polynomial.
\end{example}}

For a circuit polynomial $f=\sum_{\a\in\A}c_{\a}\x^{\a}-d\x^{\b}$, since $\b\in\Conv(\A)^{\circ}$, $\b$ admits a unique convex representation: 
$\b=\sum_{\a\in\A}\lambda_{\a}\a$ with $\lambda_{\a}>0$ and $\sum_{\a\in\A}\lambda_{\a}=1$.
Then we define the corresponding {\em circuit number} as $\Theta_f:=\prod_{\a\in\A}(c_{\a}/\lambda_{\a})^{\lambda_{\a}}$.
It is known that the nonnegativity of a circuit polynomial is decided by its circuit number alone.
\begin{theorem}(\cite[Theorem 3.8]{iw})\label{nc-thm1}
Let $f=\sum_{\a\in\A}c_{\a} \x^{\a}-d\x^{\b}\in\R[\x]$ be a circuit polynomial and $\Theta_f$ its circuit number. Then $f$ is nonnegative if and only if either $\b\in(2\N)^n$ and $d\le\Theta_f$ or $\b\notin(2\N)^n$ and $|d|\le\Theta_f$.
\end{theorem}

\begin{remark}
For the concise of narrative, we also view a monomial square as a nonnegaive circuit polynomial.
\end{remark}

The following proposition characterizes the zeros of a circuit polynomial when the Newton polytope is full-dimensional.
\begin{proposition}(\cite[Proposition 3.4 and Corollary 3.9]{iw})\label{nc-prop1}
Let $f=\sum_{i=0}^nc_i \x^{\a_i}-\Theta_f\x^{\b}\in\R[\x]$ be a circuit polynomial, $\Theta_f$ the circuit number, and $\b=\sum_{i=0}^n\lambda_i\a_i$ with $\lambda_i>0$ and $\sum_{i=0}^n\lambda_i=1$. Then $f$ has exactly one zero $\x_*$ in $\R_+^{n}$ which satisfies:
\begin{equation}\label{nc-eq1}
\frac{c_0\x_*^{\a_0}}{\lambda_0}=\cdots=\frac{c_n\x_*^{\a_n}}{\lambda_n}=\Theta_f\x_*^{\b}.
\end{equation}
Moreover, if $\x$ is any zero of $f$, then $|\x|=\x_*$, i.e. $|x_i|=(x_*)_{i}$ for $i=1,\ldots,n$.
\end{proposition}
\begin{proof}
Consider $f'=\lambda_0f/(c_0\x^{\a_0})$. Clearly, the zeros in $\R_+^{n}$ of $f$ coincide with the zeros in $\R_+^{n}$ of $f'$. By Proposition 3.4 in \cite{iw}, $f'$ and hence $f$ have exactly one zero $\x_*$ in $\R_+^{n}$ which satisfies $\x_*^{\a_i-\a_0}=(\lambda_ic_0)/(c_i\lambda_0)$ for $i=1,\ldots,n$. Let $s=(c_0\x_*^{\a_0})/\lambda_0=\cdots=(c_n\x_*^{\a_n})/\lambda_n$. Then $s=\sum_{i=0}^n\lambda_i s=\sum_{i=0}^nc_i\x_*^{\a_i}=\Theta_f\x_*^{\b}$ and so \eqref{nc-eq1} is proved. The last statement of the theorem follows from Corollary 3.9 in \cite{iw}.
\end{proof}

\begin{remark}\label{sec2:rm1}
Note that in Proposition \ref{nc-prop1}, $\x_*\in\R_+^{n}$ and the circuit number $\Theta_f$ are uniquely determined by the equations \eqref{nc-eq1}.
\end{remark}

We shall say that a polynomial is a {\em sum of nonnegative circuit polynomials (SONC)}, if it can be rewritten as a sum of nonnegative circuit polynomials. Clearly, an explicit representation of a SONC polynomial as a sum of nonnegative circuit polynomials provides a certificate of its nonnegativity, which is called a {\em SONC decomposition}. The class of SONC polynomials forms a convex cone, i.e. the {\em SONC cone}.

The following theorem from \cite{iw} adapted to our notation gives a characterization for a nonnegative polynomial to be a SONC polynomial when the Newton polytope is a simplex.
\begin{theorem}(\cite[Corollary 7.5]{iw})\label{nc-thm2}
Let $f=\sum_{i=0}^n c_i \x^{\a_i}-\sum_{j=1}^l d_j\x^{\b_j}\in\R[\x]$ be nonnegative with $\a_i\in(2\N)^n, c_i\in\R_+, i=0,\ldots,n$ such that $\New(f)$ is a simplex and $\b_j\in\New(f)^{\circ}$ for $j=1,\ldots,l$. If there exists a point $\bv=[v_k]\in(\R^*)^n$ such that $d_j\bv^{\b_j}>0$ for all $j$, then $f\in\SONC$.
\end{theorem}

\section{Nonnegative polynomials with one negative term}\label{sec3}
Now following the line of Theorem \ref{nc-thm2}, we study which types of nonnegative polynomials with general Newton polytopes are SONC polynomials. The well-known Hilbert's classification on the coincidence of nonnegative polynomials and SOS polynomials is according to the number of variables and the degree of polynomials. As to the SONC case, we will see that the related classification depends on the combinatorical structure of supports of polynomials. In this section, we deal with the case of nonnegative polynomials with one negative term (by a negative term we refer to a term that takes a negative value at some point), i.e., polynomials of the form $f_d=\sum_{i=1}^mc_i \x^{\a_i}-d\x^{\b}\in\R[\x]$ with $\a_i\in(2\N)^n,c_i\in\R_+,i=1,\ldots,m$ and $\b\notin V(\New(f_d))$. Let $\partial\New(f_d)$ denote the boundary of $\New(f_d)$. We first reduce the case of $\b\in\partial\New(f_d)$ to the case $\b\in\New(f_d)^{\circ}$ by the following lemma.

\begin{lemma}\label{sec3-lm}
Let $f_d=\sum_{i=1}^mc_i \x^{\a_i}-d\x^{\b}\in\R[\x]$ with $\a_i\in(2\N)^n,c_i\in\R_+,i=1,\ldots,m$ and $\b\in\partial\New(f_d)$. Furthermore, let $F$ be the minimal face of $\New(f_d)$ containing $\b$. Then $f_d$ is nonnegative if and only if the restriction of $f_d$ to the face $F$ is nonnegative.
\end{lemma}
\begin{proof}
The necessity follows from \cite[Theorem 3.6]{re1}. For the sufficiency, note that the restriction to the face $F$ contains the term $-d\x^{\b}$ and this restriction is nonnegative. Moreover, all other terms in $f_d$ are monomial squares. Hence $f_d$ is nonnegative.
\end{proof}

From now on, we assume $\b\in\New(f_d)^{\circ}$. Without loss of generality, we further make the assumption that the Newton polytope of $f_d$ is full-dimensional, i.e. $\dim(\New(f_d))=n$. Otherwise, we can reduce to this case by applying an appropriate monomial transformation to $f_d$ (c.f. \cite{pantea}).

To begin with, we give a characterization for $f_d\in\PSD$ as well as the positive zeros of $f_d$ in a similar manner as Theorem \ref{nc-thm1} and Proposition \ref{nc-prop1}. It turns out that $f_d$ behaves just like a circuit polynomial.

It is easy to see that the set $\{d\in\R\mid f_d\in\PSD\}$ is nonempty and has upper bounds. So the supremum exists. Let
\begin{equation}\label{npgp-eq1}
d^*\triangleq\sup\{d\in\R\mid f_d\in\PSD\}.
\end{equation}
The quantity $d^*$ is an analogy of the circuit number in the situation of $f_d$.

\begin{proposition}\label{npgp-thm3}
Let $f_d=\sum_{i=1}^mc_i \x^{\a_i}-d\x^{\b}\in\R[\x]$ with $\a_i\in(2\N)^n,c_i\in\R_+,i=1,\ldots,m$ such that $\b\in\New(f_d)^{\circ}$, $\dim(\New(f_d))=n$, and let $d^*$ be defined as (\ref{npgp-eq1}). Then $f_d\in\PSD$ if and only if either $\b\in(2\N)^n$ and $d\le d^*$ or $\b\notin(2\N)^n$ and $|d|\le d^*$. Moreover, $f_{d^*}$ has exactly one zero in $\R_+^{n}$.
\end{proposition}
\begin{proof}
First, if $\b\in(2\N)^n$ and $d\le0$, then $f_d$ is obviously nonnegative since it is a sum of monomial squares. If $\b\notin(2\N)^n$ and $d\le0$, then $f_d$ is nonnegative if and only if $f_{-d}$ is nonnegative. Thus without loss of generality, we may assume $d>0$. Since the only negative term of $f_d$ is $-d\x^{\b}$, $f_d$ is nonnegative over $\R^n$ if and only if $f_d$ is nonnegative over $\R_+^n$. Therefore, by the definition of $d^*$, $f_d\in\PSD$ if and only if $d\le d^*$. 

To prove the second statement, consider $f'=\sum_{i=1}^m c_i\x^{\a_i-\b}$. It is easy to see that $d^*=\inf_{\x\in\R^n_+}f'$. Because $\dim(\New(f_d))=n$ and $\b\in\New(f_d)^{\circ}$, we have $\dim(\Conv(\{\a_1-\b,\ldots,\a_m-\b\}))=n$ and $\mathbf{0}\in\Cone(\{\a_1-\b,\ldots,\a_m-\b\})^{\circ}$. Therefore by Theorem 3.4 in \cite{pantea}, $f'$ attains its minimum over $\R^n_+$ at a unique minimizer. Since the minimizers of $f'$ over $\R^n_+$ coincide with the zeros of $f_{d^*}$ in $\R^n_+$, it follows that $f_{d^*}$ has exactly one zero in $\R_+^{n}$.
\end{proof}

For a nonnegative $f_d=\sum_{i=1}^mc_i\x^{\a_i}-d\x^{\b}$ with $\b\in\New(f_d)^{\circ}$, let $\C$ be the set of all circuits $\A\cup\{\b\}$ with $\A\subseteq\{\a_1,\ldots,\a_m\}$. In the rest of this section, we prove that $f_d$ decomposes into a sum of nonnegative circuit polynomials that are supported on circuits in $\C$. We first consider the decomposition of $f_{d^*}$ and then get the decomposition of $f_d$ from that of $f_{d^*}$. By using undetermined coefficients, the existence of such a decomposition of $f_{d^*}$ is reduced to the existence of a nonnegative solution for a particular linear system, which can be further reduced to the existence of a nonnegative solution for a tuple of subsystems by virtue of the following result, known as Helly's theorem.

\begin{theorem}[Helly, \cite{dgk}]\label{npgp-thm5}
Let $X_1,\ldots,X_r$ be a finite collection of convex subsets of $\R^s$ with $r>s$. If the intersection of every $s+1$ of these sets is nonempty, then the whole collection has a nonempty intersection.
\end{theorem}

\begin{remark}\label{helly-rmk}
In Theorem \ref{npgp-thm5}, if $r>s+1$, then the condition that the intersection of every $s+1$ of these sets is nonempty can be replaced by the condition that the intersection of every $r-1$ of these sets is nonempty since the latter obviously implies the former.
\end{remark}

Next using Helly's theorem, we prove a result concerning the existence of nonnegative solutions to a particular class of linear systems for later use, which might be also of independent interest.

\begin{lemma}\label{npgp-lm1}
Let $A=[a_{ij}]\in\R^{m\times r}$, $\bb=[b_i]\in\R^m$ and $\z=(z_1,\ldots,z_r)^\intercal$ be a set of variables. For each $j$, let $A_j$ be the submatrix by deleting all of the $i$-th rows with $a_{ij}\ne0$ and the $j$-th column from $A$ such that $A_j\bar{\z}_j=\bar{\bb}_j$ is the subsystem of $A\z=\bb$ after removing the equations involving the variable $z_j$, where $\bar{\z}_j=\z\backslash z_j$, \revision{$\bar{\bb}_j=[b_i]_{i\textrm{ with }a_{ij}=0}$ (namely, removing the entries with $a_{ij}\neq0$ from $\bb$)}. Assume that $A\z=\bb$ is consistent, $\rank(A)>1$ and $\rank(A_j)=\rank(A)-1$ for all $j$. Then $A\z=\bb$ has a nonnegative solution if and only if $A_j\bar{\z}_j=\bar{\bb}_j$ has a nonnegative solution for $j=1,\ldots,r$.
\end{lemma}
\begin{proof}
Let $t=\rank(A)>1$. Then the system of linear equations $A\z=\bb$ has $r-t$ free variables. Without loss of generality, let the $r-t$ free variables be $\{z_1,\ldots,z_{r-t}\}$. So we can solve for $\{z_{r-t+1},\ldots,z_r\}$ from $A\z=\bb$ and assume $z_{i}=h_{i}(z_1,\ldots,z_{r-t})$ for $i=r-t+1,\ldots,r$. Then $A\z=\bb$ has a nonnegative solution if and only if the set
\revision{\begin{align}\label{npgp-eq6}
\begin{split}
\{(z_1,\ldots,z_{r-t})\in\R^{r-t}\mid &z_i\ge0\textrm{ for }i\in[r-t],\\
&z_{i}=h_{i}(z_1,\ldots,z_{r-t})\ge0\textrm{ for }i\in\{r-t+1,\ldots,r\}\}
\end{split}
\end{align}}
is nonempty. \revision{Define $X_i:=\{(z_1,\ldots,z_{r-t})\in\R^{r-t}\mid z_i\ge0\}$ for $i=1,\ldots,r-t$ and $X_i:=\{(z_1,\ldots,z_{r-t})\in\R^{r-t}\mid h_{i}(z_1,\ldots,z_{r-t})\ge0\}$ for $i=r-t+1,\ldots,r$, which are all convex subsets of $\R^{r-t}$. Therefore by Theorem \ref{npgp-thm5} as well as Remark \ref{helly-rmk}, the intersection of all $X_i$, i.e. \eqref{npgp-eq6}, is nonempty if and only if the intersection of every $r-1$ of these sets is nonempty, that is},
the set
\revision{\begin{align}\label{npgp-eq7}
\begin{split}
\{(z_1,\ldots,z_{r-t})\in\R^{r-t}\mid &z_i\ge0\textrm{ for }i\in[r-t]\backslash\{j\},\\
&z_{i}=h_{i}(z_1,\ldots,z_{r-t})\ge0\textrm{ for }i\in\{r-t+1,\ldots,r\}\}
\end{split}
\end{align}}
is nonempty for $j=1,\ldots,r-t$ and the set
\revision{\begin{align}\label{npgp-eq8}
\begin{split}
\{(z_1,\ldots,z_{r-t})\in\R^{r-t}\mid &z_i\ge0\textrm{ for }i\in[r-t],\\
&z_{i}=h_{i}(z_1,\ldots,z_{r-t})\ge0\textrm{ for }i\in\{r-t+1,\ldots,r\}\backslash\{j\}\}
\end{split}
\end{align}}
is nonempty for $j=r-t+1,\ldots,r$.

For $j=1,\ldots,r-t$, (\ref{npgp-eq7}) is nonempty if and only if $A\z=\bb$ has a solution with $\bar{\z}_j\in\R_{\ge0}^{r-1}$ and $z_j\in\R$, which is equivalent to the condition that $A_j\bar{\z}_j=\bar{\bb}_j$ has a nonnegative solution since $\rank(A_j)=\rank(A)-1$. For $j=r-t+1,\ldots,r$, (\ref{npgp-eq8}) is nonempty if and only if $A\z=\bb$ has a solution with $\bar{\z}_j\in\R_{\ge0}^{r-1}$ and $z_j\in\R$, which is also equivalent to the condition that $A_{j}\bar{\z}_{j}=\bar{\bb}_{j}$ has a nonnegative solution since $\rank(A_{j})=\rank(A)-1$. Put all above together and we deduce that $A\z=\bb$ has a nonnegative solution if and only if $A_j\bar{\z}_j=\bar{\bb}_j$ has a nonnegative solution for $j=1,\ldots,r$ as desired.
\end{proof}

\revision{
\begin{example}
Consider the linear system $S=\{z_1+z_2=1,z_3+z_4=2,z_2+z_3=1,z_1+z_4=2,z_1+z_2+z_3+z_4=3\}$. It is easy to check that $S$ satisfies the hypotheses of Lemma \ref{npgp-lm1}. Clearly all subsystems $\{z_3+z_4=2,z_2+z_3=1\},\{z_3+z_4=2,z_1+z_4=2\},\{z_1+z_2=1,z_1+z_4=2\},\{z_1+z_2=1,z_2+z_3=1\}$ admit a nonnegative solution. Thus by Lemma \ref{npgp-lm1} we conclude that $S$ has a nonnegative solution.
\end{example}}

Lemma \ref{npgp-lm1} assumes the consistentness of $A\z=\bb$. It is known that the system of linear equations $A\z=\bb$ is consistent if and only if $\bb$ belongs to the image of $A$. For later use, we give a more concrete description concerning the consistentness of $A\z=\bb$ here, whose correctness is obvious and thus we omit the proof.
\revision{\begin{lemma}\label{npgp-lm2}
Let $A=[a_{ij}]\in\R^{m\times r}$, $\bb=[b_j]\in\R^m$, and $\z=(z_1,\ldots,z_r)^\intercal$ be a set of variables. Assume that the row vectors of the matrix $C$ span the kernel of $A^\intercal$. Then $A\z=\bb$ is consistent if and only if $C\bb=\mathbf{0}$.
\end{lemma}}
% \begin{proof}
% Observe that $\{\mathbf{c}_1,\ldots,\mathbf{c}_l\}$ generates the linear space of all linear relationships among $\{\mathbf{a}_1,\ldots,\mathbf{a}_m\}$. In other words, $\{\mathbf{c}_1,\ldots,\mathbf{c}_l\}$ generates the kernel space of $A^\intercal$. Thus $\rank(C)=\rank(\ker(A^\intercal))=m-\rank(A)$.

% If $C\bb=\mathbf{0}$, i.e. $\bb$ is a zero of the elimination ideal $I$, then by the Extension Theorem in \cite[p.\,125]{cls}, we can extend $\bb$ to a zero of the ideal $(\mathbf{a}_1^\intercal\z-y_1,\ldots,\mathbf{a}_m^\intercal\z-y_m)$. Therefore $A\z=\bb$ is consistent. The converse is easy.
% \end{proof}

% For $\A=\{\a_i\}_{i=1}^m\subseteq\N^n$, $\A'=\{\a'_i\}_{i=1}^m\subseteq\Z^n$, $\{c_i\}_{i=1}^m\subseteq\R_+$ and $\bb\in\R^n$, the system of polynomial equations $\sum_{i=1}^mc_i\a'_i\x^{\a_i}=\bb$ consists of the $n$ equations: $$\sum_{i=1}^mc_i\alpha'_{ij}\x^{\a_i}=b_j,\quad \textrm{for } j=1,\ldots,n.$$

Now we are ready to prove $f_{d^*}\in\SONC$.
\begin{lemma}\label{npgp-thm4}
Let $f_d=\sum_{i=1}^mc_i\x^{\a_i}-d\x^{\b}\in\R[\x]$ with $\a_i\in(2\N)^n,c_i\in\R_+,i=1,\ldots,m$ such that $\b\in\New(f_d)^{\circ}$, $\dim(\New(f_d))=n$, and let $d^*$ be defined as (\ref{npgp-eq1}). Then $f_{d^*}\in\SONC$.
\end{lemma}
\begin{proof}
If $m=n+1$, then $f_d$ is a circuit polynomial and clearly $f_{d^*}\in\SONC$. From now on, we assume $m>n+1$. By Proposition \ref{npgp-thm3}, $f_{d^*}$ has exactly one zero in $\R_+^{n}$, which is denoted by $\x_*$. Let
$$\{\Delta_1,\ldots,\Delta_r\}:=\{\Delta\mid\Delta\textrm{ is a simplex }, \b\in\Delta^{\circ}, V(\Delta)\subseteq\{\a_1,\ldots,\a_m\}\}$$
and $I_k:=\{i\in[m]\mid\a_i\in V(\Delta_k)\}$ for $k=1,\ldots,r$.
\revision{We complete the proof by constructing a SONC decomposition supported on the simplices $\{\Delta_k\}_k$ for $f_{d^*}$.} 

Firstly, we assume $\dim(\Delta_k)=n$ so that $|I_k|=n+1$ for $k=1,\ldots,r$. For each $\Delta_k$, since $\b\in\Delta_k^{\circ}$, we can write $\b=\sum_{i\in I_k}\lambda_{ik}\a_i$, where $\sum_{i\in I_k}\lambda_{ik}=1, \lambda_{ik}>0, i\in I_k$. Inspired by Proposition \ref{nc-prop1} and using undetermined coefficients, we may consider the following system of linear equations in variables $\{c_{ik}\}_{i,k}$ and $\{s_k\}_k$:
\begin{equation}\label{npgp-eq2}
\begin{cases}
\frac{c_{ik}\x_*^{\a_i}}{\lambda_{ik}}=s_k, &\textrm{for }i\in I_k,k=1,\ldots,r,\\
\sum_{k\in[r]\textrm{ with }i\in I_k} c_{ik}=c_i, &\textrm{for }i=1,\ldots,m.
\end{cases}
\end{equation}
Eliminate the variables $\{c_{ik}\}_{i,k}$ from (\ref{npgp-eq2}) and we obtain:
\begin{equation}\label{npgp-eq3}
\sum_{k\in[r]\textrm{ with }i\in I_k}\lambda_{ik}s_k=c_i\x_*^{\a_i}, \quad \textrm{for }i=1,\ldots,m.
\end{equation}
If \eqref{npgp-eq3} has a nonnegative solution, then we can retrieve a SONC decomposition for $f_{d^*}$ from the nonnegative solution as follows. Assume that $\{s_1^*,\ldots,s_r^*\}$ is a nonnegative solution to \eqref{npgp-eq3}. Substitute $\{s_1^*,\ldots,s_r^*\}$ into the system of equations \eqref{npgp-eq2}, and we have $c_{ik}=\lambda_{ik}s_k^*/\x_*^{\a_i}$ for $i\in I_k,k=1,\ldots,r$. Let $d_k=s_k^*/\x_*^{\b}$ and $f_k=\sum_{i\in I_k}c_{ik}\x^{\a_i}-d_k\x^{\b}$ for $k=1,\ldots,r$. Then by (\ref{npgp-eq2}) and by Proposition \ref{nc-prop1}, $d_{k}$ is the circuit number of $f_{k}$ and hence $f_k$ is a nonnegative circuit polynomial for all $k$. By (\ref{npgp-eq2}), $\sum_{k=1}^rd_k\x_*^{\b}=\sum_{k=1}^r\sum_{i\in I_k}c_{ik}\x^{\a_i}_*=\sum_{i=1}^mc_i \x^{\a_i}_*=d^*\x_*^{\b}$, which implies $\sum_{k=1}^rd_k=d^*$. It follows $f_{d^*}=\sum_{k=1}^rf_k\in\SONC$ as desired. \revision{So our remaining task is to prove that \eqref{npgp-eq3} has a nonnegative solution.}

{\bf Claim}: The linear system (\ref{npgp-eq3}) in variables $\{s_1,\ldots,s_r\}$ has a nonnegative solution.

{\em Proof of the claim}. Denote the coefficient matrix of (\ref{npgp-eq3}) by \revision{$A=[a_{ik}]\in\R^{m\times r}$ (satisfying $a_{ik}=\lambda_{ik}$ if $i\in I_k$ and $a_{ik}=0$ otherwise)} and denote the coefficient matrix of
\begin{equation}\label{npgp-eq4}
\sum_{k\in[r]\textrm{ with }i\in I_k}\lambda_{ik}s_k=c_i\x_*^{\a_i},\quad \textrm{for } i\in[m]\backslash I_j
\end{equation}
by $A_j$ for each $j\in[r]$. Note that \eqref{npgp-eq4} is obtained from \eqref{npgp-eq3} by removing the equations involving the variable $s_j$.
In order to invoke Lemma \ref{npgp-lm1} to prove that (\ref{npgp-eq3}) has a nonnegative solution, we need to check the following hypotheses:
\begin{enumerate}
    \item $\rank(A)>1$;
    \item $\rank(A_j)=\rank(A)-1$ for each $j\in[r]$;
    \item (\ref{npgp-eq3}) is consistent.
\end{enumerate}

\revision{Fix $j\in[r]$. For every $i\in[m]\backslash I_j$, since $\b\in\Delta_j^{\circ}$, there exists a facet $F$ of $\Delta_j$ such that $\b\in\Conv(V(F)\cup\{\a_i\})^{\circ}$. Let $\Conv(V(F)\cup\{\a_i\})=\Delta_{p_i}$ for some $p_i\in[r]$ (see Figure \ref{fg1}). It is easy to see that $p_{i_1}\ne p_{i_2}$ whenever $i_1\ne i_2$.
\begin{figure}
\caption{Illustration for the correspondence between $\a_i$ and $\Delta_{p_i}$ for $i\in[m]\backslash I_j$}\label{fg1}
\begin{center}
\begin{tikzpicture}
\fill (90:2) circle (2pt);
\fill (162:2) circle (2pt);
\fill (234:2) circle (2pt);
\fill (306:2) circle (2pt);
\fill (18:2) circle (2pt);
\fill (260:1.3) circle (2pt);
\node[above left] (1) at (162:2) {$\a_i$};
\node (2) at (90:0) {$\Delta_j$};
\node[left] (3) at (245:0.9) {$\Delta_{p_i}$};
\node[right] (4) at (260:1.3) {$\b$};
\node[below] (5) at (270:1.6) {$F$};
\draw (90:2)--(162:2);
\draw (90:2)--(234:2);
\draw (90:2)--(306:2);
\draw (162:2)--(306:2);
\draw (162:2)--(234:2);
\draw (234:2)--(306:2);
\draw (306:2)--(18:2);
\draw (18:2)--(90:2);
\fill[fill=green,fill opacity=0.3] (90:2)--(234:2)--(306:2);
\fill[fill=blue,fill opacity=0.3] (306:2)--(162:2)--(234:2);
\end{tikzpicture}
\end{center}
\end{figure} 
For every $k\in[r]\backslash(\{j\}\cup\{p_i\mid i\in[m]\backslash I_j\})$, let $s_k=0$ in (\ref{npgp-eq4}) and then by construction we obtain:
\begin{equation}\label{npgp-eq5}
\lambda_{ip_i}s_{p_i}=c_i\x_*^{\a_i},\quad \textrm{for } i\in[m]\backslash I_j.
\end{equation}
It follows that $\rank(A_j)=m-|I_j|=m-(n+1)$ and furthermore, $\rank(A)\ge\rank(A_j)+1=m-n$, $\rank(\ker(A^\intercal))=m-\rank(A)\le n$. Let $C:=[\a_1-\b,\ldots,\a_m-\b]$. Then, 
% First we derive the matrix $C$ in Lemma \ref{npgp-lm2}. 
\begin{align*}
    CA&=[\sum_{i=1}^m(\a_i-\b)a_{i1},\ldots,\sum_{i=1}^m(\a_i-\b)a_{ir}]\\
    &=[\sum_{i\in I_1}(\a_i-\b)\lambda_{i1},\ldots,\sum_{i\in I_r}(\a_i-\b)\lambda_{ir}]\\
    &=[\sum_{i\in I_1}\lambda_{i1}\a_i-\b,\ldots,\sum_{i\in I_r}\lambda_{ir}\a_i-\b]=[\mathbf{0},\ldots,\mathbf{0}].
\end{align*}
So the row vectors of $C$ belong to the kernel of $A^\intercal$. Because $\b\in\New(f_d)^{\circ}$ and $\dim(\New(f_d))=n$, $\rank(C)=\rank(\{\a_i-\b\}_{i=1}^m)=n$. As $\rank(\ker(A^\intercal))\le n$, we then conclude that $\rank(\ker(A^\intercal))=n$ and the row vectors of $C$ span the kernel of $A^\intercal$. As a result, $\rank(A)=m-n>1$ and $\rank(A_j)=\rank(A)-1$.}
% Add up all the equations of (\ref{npgp-eq3}) and we obtain:
% \begin{equation}\label{npgp-eq9}
% \sum_{k=1}^rs_k=\sum_{i=1}^m\sum_{k\in[r]\textrm{ with }i\in I_k}\lambda_{ik}s_k=\sum_{i=1}^mc_i\x_*^{\a_i}.
% \end{equation}
% Multiply the $i$-th equation of (\ref{npgp-eq3}) by $\a_i$ and then add up all of them. Consequently, we obtain:
% \begin{equation}\label{npgp-eq10}
% \b\sum_{k=1}^rs_k=\sum_{i=1}^m\sum_{k\in[r]\textrm{ with }i\in I_k}\lambda_{ik}\a_is_k=\sum_{i=1}^mc_i\a_i\x_*^{\a_i}.
% \end{equation}
% Subtracting $(\ref{npgp-eq9})\times\b$ from $(\ref{npgp-eq10})$ yields
% \begin{equation}\label{npgp-eq11}
% \sum_{i=1}^m(\a_i-\b)\cdot c_i\x_*^{\a_i}=\mathbf{0}.
% \end{equation}
% Therefore with $\{\a_1-\b,\ldots,\a_m-\b\}$ as column vectors, we get the matrix $C$ in Lemma \ref{npgp-lm2}.
Because the zero $\x_*$ is also a minimizer of $f_{d^*}$, it satisfies $\{f_{d^*}(\x_*)=0,\nabla f_{d^*}(\x_*)=\mathbf{0}\}$ ($\nabla$ denotes the gradient with respect to $\x$) which gives
\begin{equation}\label{sec1-eq5}
\begin{cases}
\sum_{i=1}^m c_i\x_*^{\a_i}-d^*\x_*^{\b}=0,\\
\sum_{i=1}^m c_i\a_i\x_*^{\a_i}-d^*\b\x_*^{\b}=\mathbf{0}.
\end{cases}
\end{equation}
It follows $\sum_{i=1}^m c_i(\a_i-\b)\x_*^{\a_i}=\mathbf{0}$, i.e. $C\cdot[c_1\x_*^{\a_1},\ldots,c_m\x_*^{\a_m}]^{\intercal}=\mathbf{0}$. Thus by Lemma \ref{npgp-lm2}, (\ref{npgp-eq3}) is consistent. 

Now by Lemma \ref{npgp-lm1}, in order to prove the claim, we only need to show that every subsystem (\ref{npgp-eq4}) in variables $\{s_1,\ldots,s_r\}\backslash\{s_j\}$ has a nonnegative solution for $j=1,\ldots,r$. Given $j\in[r]$, from (\ref{npgp-eq5}) we figure out $s_{p_i}=c_i\x_*^{\a_i}/\lambda_{ip_i}$ for $i\in[m]\backslash I_j$. Hence
\begin{equation}
\begin{cases}
s_k=0, &\textrm{for }k\in[r]\backslash(\{j\}\cup\{p_i\mid i\in[m]\backslash I_j\}),\\
s_{p_i}=c_i\x_*^{\a_i}/\lambda_{ip_i}, &\textrm{for }i\in[m]\backslash I_j
\end{cases}
\end{equation}
is a nonnegative solution to (\ref{npgp-eq4}). So the claim is proved.

For the case that $\dim(\Delta_k)=n$ does not hold for all $k$, note that all results above remain valid for $\b\in\R^n$. We then give $\b$ a small perturbation, say $\dd$, such that $\dim(\Delta_k)=n$ holds for all $k$. Then the new linear system (\ref{npgp-eq3}) for $\b+\dd$ has a nonnegative solution. Let $\dd\to\mathbf{0}$. We obtain that (\ref{npgp-eq3}) also has a nonnegative solution for $\b$. Thus the theorem remains true in this case.
\end{proof}

\revision{We give an example to illustrate Lemma \ref{npgp-thm4}.
\begin{example}
Let $f_d=1+x^4+y^4+x^6y^4+x^4y^6-dx^2y$ and $d^*=\sup\{d\in\R_+\mid f_d\in\PSD\}$. We have $[2,1]^\intercal=\frac{1}{4}[0,0]^\intercal+\frac{1}{2}[4,0]^\intercal+\frac{1}{4}[0,4]^\intercal=\frac{1}{2}[0,0]^\intercal+\frac{1}{3}[4,0]^\intercal+\frac{1}{6}[4,6]^\intercal=\frac{5}{8}[0,0]^\intercal+\frac{1}{8}[4,0]^\intercal+\frac{1}{4}[6,4]^\intercal$.
\begin{center}
\begin{tikzpicture}
\draw (0,0)--(0,2.4);
\draw (0,0)--(2.4,0);
\draw (2.4,0)--(3.6,2.4);
\draw (0,2.4)--(2.4,3.6);
\draw (3.6,2.4)--(2.4,3.6);
\draw (0,2.4)--(2.4,0);
\draw (2.4,0)--(2.4,3.6);
\draw (0,0)--(3.6,2.4);
\draw (0,0)--(2.4,3.6);
\fill (0,0) circle (2pt);
\node[below left] (1) at (0,0) {$1$};
\fill (2.4,0) circle (2pt);
\node[below right] (2) at (2.4,0) {$x^4$};
\fill (0,2.4) circle (2pt);
\node[above left] (3) at (0,2.4) {$y^4$};
\fill (2.4,3.6) circle (2pt);
\node[above right] (4) at (2.4,3.6) {$x^4y^6$};
\fill (3.6,2.4) circle (2pt);
\node[above right] (5) at (3.6,2.4) {$x^6y^4$};
\fill (1.2,0.6) circle (2pt);
\node[below] (6) at (1.2,0.6) {$x^2y$};
\node (7) at (0.7,1) {$\Delta_1$};
\node (8) at (1.5,1.5) {$\Delta_2$};
\node (9) at (2.3,1) {$\Delta_3$};
\fill[fill=green,fill opacity=0.3] (0,0)--(0,2.4)--(2.4,0);
\fill[fill=blue,fill opacity=0.3] (0,0)--(2.4,0)--(2.4,3.6);
\fill[fill=yellow,fill opacity=0.3] (0,0)--(2.4,0)--(3.6,2.4);
\end{tikzpicture}
\end{center}
The system of equations $\{f_d=0,\nabla f_d=\mathbf{0}\}$ in variables $\{x,y,d\}$ has exactly one zero $(x_*\approx0.944112,y_*\approx0.708568,d^*\approx3.682248)$ in $\R_+^{3}$. The linear system \eqref{npgp-eq3} becomes
\begin{equation}\label{npgp-ex}
\begin{cases}
\frac{1}{4}s_{1}+\frac{1}{2}s_{2}+\frac{5}{8}s_{3}=1\\
\frac{1}{2}s_{1}+\frac{1}{3}s_{2}+\frac{1}{8}s_{3}=x_*^4\\
\frac{1}{4}s_{1}=y_*^4\\
\frac{1}{4}s_{3}=x_*^6y_*^4\\
\frac{1}{6}s_{2}=x_*^4y_*^6\\
\end{cases}
\end{equation}
which has a nonnegative solution $(s_{1}\approx1.00829,s_{2}\approx0.603299,s_{3}\approx0.714045)$. Thus from the proof of Lemma \ref{npgp-thm4}, we obtain a SONC decomposition of $f_{d^*}$ which is $f_{d^*}\approx(0.252072+0.634543x^4+y^4-1.59646x^2y)+(0.30165+0.253115x^4+x^4y^6-0.955222x^2y)+(0.446278+0.112342x^4+x^6y^4-1.13057x^2y)$.
\end{example}}

\begin{theorem}\label{npgp-thm7}
Let $f_d=\sum_{i=1}^mc_i \x^{\a_i}-d\x^{\b}\in\R[\x]$ with $\a_i\in(2\N)^n,c_i\in\R_+,i=1,\ldots,m$ such that $\b\in\New(f_d)^{\circ}$, $\dim(\New(f_d))=n$. Then $f_d\in\PSD$ if and only if $f_d\in\SONC$.
\end{theorem}
\begin{proof}
The sufficiency is obvious. Assume that $f_d$ is nonnegative. If $\b\in(2\N)^n$ and $d<0$, or $d=0$, then $f_d$ is a sum of monomial squares and obviously $f_d\in\SONC$. If $\b\notin(2\N)^n$ and $d<0$, through a variable transformation $x_j\mapsto -x_j$ for some odd number $\beta_j$, we can always assume $d>0$. Let $d^*$ be defined as (\ref{npgp-eq1}). By Lemma \ref{npgp-thm4} and its proof, $f_{d^*}\in\SONC$ and $f_{d^*}$ admits a SONC decomposition:  $f_{d^*}=\sum_{k=1}^r(\sum_{i\in I_k}c_{ik}\x^{\a_i}-d_k\x^{\b})$, where $\sum_{i\in I_k}c_{ik}\x^{\a_i}-d_k\x^{\b}$ is a nonnegative circuit polynomial with $d_k$ the corresponding circuit number for all $k$ (the sets $I_k,k\in[r]$ are defined in the proof of Lemma \ref{npgp-thm4}). Since $f_d$ is nonnegative, it follows $d\le d^{*}$. We have $f_d=\sum_{k=1}^r(\sum_{i\in I_k}c_{ik}\x^{\a_i}-\frac{d}{d^*}d_k\x^{\b})$, where $\sum_{i\in I_k}c_{ik}\x^{\a_i}-\frac{d}{d^*}d_k\x^{\b}$ is a nonnegative circuit polynomial for all $k$ by Theorem \ref{nc-thm1}. Thus $f_d\in\SONC$.
\end{proof}

\begin{remark}
Theorem \ref{npgp-thm7} is a generalization of Theorem \ref{nc-thm2} to the case of polynomials with general Newton polytopes and with a unique negative term. We point out that a special case of Theorem \ref{npgp-thm7} concerning agiforms was proved by Reznick in 1989; see \cite[Theorem 7.1]{re1}.
\end{remark}

\revision{\begin{definition}
An {\em AGE polynomial} is a nonnegative polynomial with at most one negative term, namely, it is nonnegative and of the form
\begin{equation*}
    \sum_{i=1}^mc_i \x^{\a_i}-d\x^{\b}, \textrm{ where }\a_i\in(2\N)^n,c_i\in\R_{+},i=1,\ldots,m,
\end{equation*}
and either $\b\in\N^n\backslash(2\N)^n$, or $\b\in(2\N)^n$ and $d\ge0$.
\end{definition}}

The proof of Theorem \ref{npgp-thm7} enables us to give a SONC decomposition {\em without cancellation} for AGE polynomials.
\begin{theorem}\label{npgp-thm8}
Let $f=\sum_{i=1}^mc_i \x^{\a_i}-d\x^{\b}\in\R[\x]$ with $\a_i\in(2\N)^n,c_i\in\R_+,i=1,\ldots,m$ be an AGE polynomial. Let
$$\mathscr{F}:=\{\Delta\mid\Delta\textrm{ is a simplex }, \b\in\Delta^{\circ}, V(\Delta)\subseteq\{\a_1,\ldots,\a_m\}\}.$$
Then $f$ admits a SONC decomposition as follows:
\begin{equation}
    f=\sum_{\Delta\in\mathscr{F}}f_{\Delta}+\sum_{i\in I}c_i\x^{\a_i},
\end{equation}
where $f_{\Delta}$ is a nonnegative circuit polynomial supported on $V(\Delta)\cup\{\b\}$ for each $\Delta$ and $I=\{i\in[m]\mid\a_i\notin\bigcup_{\Delta\in\mathscr{F}}V(\Delta)\}$.
\end{theorem}
\begin{proof}
It follows easily from Lemma \ref{sec3-lm} and the proof of Theorem \ref{npgp-thm7}.
\end{proof}

Murray, Chandrasekaran and Wiermann proposed the {\em sum of AGE polynomials} (SAGE) as a new nonnegativity certificate of polynomials in \cite{murray}, where they considered not only polynomial nonnegativity but also signomial nonnegativity. Nonnegative polynomials that admit a SAGE decomposition are called {\em SAGE polynomials}. The cone containing all SAGE polynomials is called the {\em SAGE cone}.
Due to Theorem \ref{npgp-thm8}, we immediately obtain the following result.
\begin{corollary}
The SONC cone coincides with the SAGE cone.
\end{corollary}

The coincidence of the SONC cone and the SAGE cone was also independently proved in \cite{murray} by showing that any extreme ray of the SAGE cone is a nonnegative circuit polynomial \revision{\cite[Corollary 21]{murray}. The proof in \cite{murray} was provided in the context of signomials and stems from convex duality. In contrast, our proof uses algebraic techniques and exploits combinatorical structure of the polynomial support in an essential way.}

\section{Nonnegative polynomials with multiple negative terms}\label{sec4}
In this section, we deal with the case of nonnegative polynomials with multiple negative terms. We will provide sufficient conditions under which a nonnegative polynomial with multiple negative terms admits a SONC decomposition. The proof proceeds in a similar manner as the proof of Theorem \ref{npgp-thm7}. We first consider the case that the polynomial lies on the boundary of the PSD cone since the general case will be reduced to this case. As in the proof of Lemma \ref{npgp-thm4}, by using undetermined coefficients, the existence of such a decomposition is reduced to the existence of a nonnegative solution for a particular linear system, which is then further reduced to the existence of a nonnegative solution for a tuple of subsystems by Lemma \ref{npgp-lm1}.

To state the theorem, we need a technical condition on the Newton polytope. Let $\Delta$ be a polytope of dimension $d$. For a vertex $\a$ of $\Delta$, we say that $\Delta$ is {\em simple} at $\a$ if $\a$ is the intersection of precisely $d$ edges.
\begin{theorem}\label{npmt-thm1}
Let $f=\sum_{i=1}^mc_i \x^{\a_i}-\sum_{j=1}^ld_j\x^{\b_j}\in\R[\x]$ with $\a_i\in(2\N)^n,c_i\in\R_+,i=1,\ldots,m$, $\b_j\in\New(f)^{\circ},j=1,\ldots,l$. Assume that $\New(f)$ is simple at some vertex, all $\b_j$ lie in the same side of every hyperplane determined by points among $\{\a_1,\ldots,\a_m\}$, and there exists a point $\bv=[v_k]\in(\R^*)^n$ such that $d_j\bv^{\b_j}>0$ for all $j$. Then $f\in\PSD$ if and only if $f\in\SONC$.
\end{theorem}
\begin{proof}
First assume $\dim(\New(f))=n$ (so $m\ge n+1$). Otherwise, we can reduce to this case by applying an appropriate monomial transformation to $f$. If $l=1$, then the conclusion follows from Theorem \ref{npgp-thm7}. From now on, we assume $l>1$. The sufficiency is obvious. Suppose $f\in\PSD$. After a variable transformation $x_k\mapsto -x_k$ for all $k$ with $v_k<0$, we can assume $d_j>0$ for all $j$. Let
\begin{equation}\label{npmt-eq9}
d_l^*\triangleq\sup\{\tilde{d}_l\in\R\mid \tilde{f}=\sum_{i=1}^mc_i \x^{\a_i}-\sum_{j=1}^{l-1}d_j\x^{\b_j}-\tilde{d}_l\x^{\b_l}\in\PSD\}.
\end{equation}
Note that $d_l^*$ is well-defined since the set in (\ref{npmt-eq9}) is nonempty and has upper bounds. Let $f^*=\sum_{i=1}^mc_i \x^{\a_i}-\sum_{j=1}^{l-1}d_j\x^{\b_j}-d_l^*\x^{\b_l}$. Then $f^*=0$ has a zero in $\R_+^{n}$ (\cite[Lemma 4.2]{wang5}), which is denoted by $\x_*$. The assumption that all $\b_j$ lie in the same side of every hyperplane determined by points among $\{\a_1,\ldots,\a_m\}$ implies if a simplex $\Delta$ with vertices coming from $\{\a_1,\ldots,\a_m\}$ contains some $\b_j$, then $\dim(\Delta)=n$ and it contains all $\b_j$. Let
$$\{\Delta_1,\ldots,\Delta_r\}:=\{\Delta\mid\Delta\textrm{ is a simplex }, \b_j\in\Delta^{\circ}, j\in[l], V(\Delta)\subseteq\{\a_1,\ldots,\a_m\}\}$$
and $I_k:=\{i\in[m]\mid\a_i\in V(\Delta_k)\}$ for $k=1,\ldots,r$. We have $\dim(\Delta_k)=n$ for all $k$. For every $\b_j$ and every $\Delta_k$, since $\b_j\in\Delta_k^{\circ}$, we can write $\b_j=\sum_{i\in I_k}\lambda_{ijk}\a_i$, where $\sum_{i\in I_k}\lambda_{ijk}=1, \lambda_{ijk}>0, i\in I_k$. In a similar manner as we prove Lemma \ref{npgp-thm4}, let us consider the following system of linear equations in variables $\{c_{ijk}\}_{i,j,k}$, $\{d_{jk}\}_{j,k}$ and $\{s_{jk}\}_{j,k}$:
\begin{equation}\label{npmt-eq1}
\begin{cases}
\frac{c_{ijk}\x_*^{\a_i}}{\lambda_{ijk}}=d_{jk}\x_*^{\b_j}=s_{jk}, &\textrm{for }i\in I_k,k=1,\ldots,r,j=1,\ldots,l,\\
\sum_{k=1}^rd_{jk}=d_j, &\textrm{for }j=1,\ldots,l-1,\\
\sum_{k=1}^rd_{lk}=d_l^*,\\
\sum_{j=1}^l\sum_{k\in[r]\textrm{ with }i\in I_k} c_{ijk}=c_i, &\textrm{for }i=1,\ldots,m.
\end{cases}
\end{equation}
Eliminate the variables $\{c_{ijk}\}_{i,j,k}$ and $\{d_{jk}\}_{j,k}$ from (\ref{npmt-eq1}) and we obtain:
\begin{equation}\label{npmt-eq2}
\begin{cases}
\sum_{j=1}^l\sum_{k\in[r]\textrm{ with }i\in I_k}\lambda_{ijk}s_{jk}=c_i\x_*^{\a_i}, &\textrm{for }i=1,\ldots,m,\\
\sum_{k=1}^rs_{jk}=d_j\x_*^{\b_j}, &\textrm{for }j=1,\ldots,l-1,\\
\sum_{k=1}^rs_{lk}=d_l^*\x_*^{\b_l}.
\end{cases}
\end{equation}
\revision{If \eqref{npmt-eq2} has a nonnegative solution, then we can retrieve a SONC decomposition supported on the simplices $\{\Delta_k\}_k$ for $f^{*}$ as follows.} Assume that $\{s_{jk}^*\}_{j,k}$ is a nonnegative solution to (\ref{npmt-eq2}). Substitute $\{s_{jk}^*\}_{j,k}$ into the system of equations (\ref{npmt-eq1}), and we have $c_{ijk}=\lambda_{ijk}s_{jk}^*/\x_*^{\a_i}$ for $i\in I_k, k=1,\ldots,r,j=1,\ldots,l$. Let $f_{jk}=\sum_{i\in I_k}c_{ijk}\x^{\a_i}-d_{jk}\x^{\b_j}$ for $k=1,\ldots,r,j=1,\ldots,l-1$. Then by (\ref{npmt-eq1}) and by Proposition \ref{nc-prop1}, $d_{jk}$ is the circuit number of $f_{jk}$ and $f_{jk}$ is a nonnegative circuit polynomial for all $j,k$. By (\ref{npmt-eq1}), we have $f=\sum_{j=1}^{l-1}\sum_{k=1}^rf_{jk}+\sum_{k=1}^r(\sum_{i\in I_k}c_{ilk}\x^{\a_i}-\frac{d_l}{d_l^*}d_{lk}\x^{\b_l})$. Since $d_l\le d_l^*$, $\sum_{i\in I_k}c_{ilk}\x^{\a_i}-\frac{d_l}{d_l^*}d_{lk}\x^{\b_l}$ is a nonnegative circuit polynomial for all $k$ by Theorem \ref{nc-thm1}. Thus $f\in\SONC$ as desired. \revision{Our remaining task hence is to prove the following claim.}

{\bf Claim}: The linear system (\ref{npmt-eq2}) in variables $\{s_{jk}\}_{j,k}$ has a nonnegative solution.

{\em Proof of the claim}. Denote the coefficient matrix of (\ref{npmt-eq2}) by $A$ and denote the coefficient matrix of
\begin{equation}\label{npmt-eq3}
\begin{cases}
\sum_{j=1}^l\sum_{k\in[r]\textrm{ with }i\in I_k}\lambda_{ijk}s_{jk}=c_i\x_*^{\a_i}, &\textrm{for }i\in[m]\backslash I_v,\\
\sum_{k=1}^rs_{jk}=d_j\x_*^{\b_j},&\textrm{for }j\in[l-1]\backslash\{u\},\\
\sum_{k=1}^rs_{lk}=d_l^*\x_*^{\b_l},&\textrm{if }u\ne l
\end{cases}
\end{equation}
by $A_{uv}$ for every $u\in[l]$ and every $v\in[r]$. Note that \eqref{npmt-eq3} is obtained from \eqref{npmt-eq2} by removing the equations involving the variable $s_{uv}$. In order to invoke Lemma \ref{npgp-lm1} to prove that (\ref{npmt-eq2}) has a nonnegative solution, we need to check the following hypotheses:
\begin{enumerate}
    \item $\rank(A)>1$;
    \item $\rank(A_{uv})=\rank(A)-1$ for every $u\in[l]$ and every $v\in[r]$;
    \item \eqref{npmt-eq2} is consistent.
\end{enumerate}

\revision{Fix $u\in[l]$ and $v\in[r]$. For every $i\in[m]\backslash I_v$, since $\b_u\in\Delta_v^{\circ}$, there exists a facet $F$ of $\Delta_v$ such that $\b_u\in\Conv(V(F)\cup\{\a_i\})^{\circ}$. Let $\Conv(V(F)\cup\{\a_i\})=\Delta_{p_i}$ for some $p_i\in[r]$. It holds $p_{i_1}\ne p_{i_2}$ whenever $i_1\ne i_2$. For every pair $(j,k)$ such that $j=u,k\in[r]\backslash(\{v\}\cup\{p_i\mid i\in[m]\backslash I_v\})$ or $j\in[l]\backslash\{u\},k\in[r]\backslash\{v\}$, let $s_{jk}=0$ in (\ref{npmt-eq3}), and then we obtain:
\begin{equation}\label{npmt-eq4}
\begin{cases}
\lambda_{iup_i}s_{up_i}=c_i\x_*^{\a_i}, &\textrm{for }i\in[m]\backslash I_v,\\
s_{jv}=d_j\x_*^{\b_j}, &\textrm{for }j\in[l-1]\backslash\{u\},\\
s_{lv}=d_l^*\x_*^{\b_l},&\textrm{if }u\ne l.
\end{cases}
\end{equation}
It follows that $A_{uv}$ has full rank and $\rank(A_{uv})=m-|I_v|+l-1=m+l-(n+2)$. Moreover, $\rank(A)\ge\rank(A_{uv})+1=m+l-(n+1)$ and $\rank(\ker(A^{\intercal}))=m+l-\rank(A)\le n+1$.
Let $C:=\left[\begin{bmatrix}1\\ \a_1\end{bmatrix},\ldots,\begin{bmatrix}1\\ \a_m\end{bmatrix},\begin{bmatrix}-1\\ -\b_1\end{bmatrix},\ldots,\begin{bmatrix}-1\\ -\b_l\end{bmatrix}\right]$. Then,
\begin{align*}
    CA=\left[\sum_{j=1}^l\left(\sum_{i\in I_1}\begin{bmatrix}1\\ \a_i\end{bmatrix}\lambda_{ij1}-\begin{bmatrix}1\\ \b_j\end{bmatrix}\right),\ldots,\sum_{j=1}^l\left(\sum_{i\in I_r}\begin{bmatrix}1\\ \a_i\end{bmatrix}\lambda_{ijr}-\begin{bmatrix}1\\ \b_j\end{bmatrix}\right)\right]=[\mathbf{0},\ldots,\mathbf{0}],
\end{align*}
which implies that the row vectors of $C$ belong to the kernel of $A^\intercal$.
Since $\dim(\Delta_1)=n$, the volume of $\Delta_1$, which equals $\frac{1}{n!}|\det(D)|$ where $D$ is the matrix with column vectors $\begin{bmatrix}1\\ \a_i\end{bmatrix}$, $i\in I_1$, is nonzero. It follows that $\rank(C)=n+1$. As $\rank(\ker(A^\intercal))\le n+1$, we then conclude that $\rank(\ker(A^\intercal))=n+1$ and the row vectors of $C$ span the kernel of $A^\intercal$. As a result, $\rank(A)=m+l-(n+1)>1$ and $\rank(A_{uv})=\rank(A)-1$. The zero $\x_*$ of $f^*$ is also a minimizer of $f^*$. So it satisfies $\{f^*(\x_*)=0,\nabla f^*(\x_*)=\mathbf{0}\}$, which gives
\begin{equation}\label{npmt-eq8}
\begin{cases}
\sum_{i=1}^mc_i\x_*^{\a_i}-\sum_{j=1}^{l-1}d_j\x_*^{\b_j}-d_l^*\x_*^{\b_l}=0,\\
\sum_{i=1}^mc_i\a_i\x_*^{\a_i}-\sum_{j=1}^{l-1}d_j\b_j\x_*^{\b_j}-d_l^*\b_l\x_*^{\b_l}=\mathbf{0},
\end{cases}
\end{equation}
i.e., $C\cdot[c_1\x_*^{\a_1},\ldots,c_m\x_*^{\a_m},d_1\x_*^{\b_1},\ldots,d_{l-1}\x_*^{\b_{l-1}},d_l^*x_*^{\b_l}]^{\intercal}=\mathbf{0}$. Thus by Lemma \ref{npgp-lm2}, (\ref{npmt-eq2}) is consistent.}

Now by Lemma \ref{npgp-lm1}, in order to prove the claim, we only need to show that every subsystem (\ref{npmt-eq3}) in variables $\{s_{jk}\}_{j,k}\backslash\{s_{uv}\}$ has a nonnegative solution for all $u\in[l]$ and all $v\in[r]$.
Given $u\in[l]$ and $v\in[r]$, from (\ref{npmt-eq4}) we figure out $s_{up_i}=c_i\x_*^{\a_i}/\lambda_{iup_i}$ for $i\in[m]\backslash I_v$, $s_{jv}=d_j\x_*^{\b_j}$ for $j\in[l-1]\backslash\{u\}$, and $s_{lv}=d_l^*\x_*^{\b_l}$ if $u\ne l$. Hence
\begin{equation*}
\begin{cases}
s_{jk}=0, &\textrm{for }j=u,k\in[r]\backslash(\{v\}\cup\{p_i\mid i\in[m]\backslash I_v\})\textrm{ or }j\in[l]\backslash\{u\},k\in[r]\backslash\{v\},\\
s_{up_i}=c_i\x_*^{\a_i}/\lambda_{iup_i}, &\textrm{for }i\in[m]\backslash I_v,\\
s_{jv}=d_j\x_*^{\b_j}, &\textrm{for }j\in[l-1]\backslash\{u\},\\
s_{lv}=d_l^*\x_*^{\b_l},&\textrm{if }u\ne l
\end{cases}
\end{equation*}
is a nonnegative solution to (\ref{npmt-eq3}). So the claim is proved and the proof is completed.
\end{proof}

\begin{remark}\label{rm}
When $\dim(\New(f))=n$ and $m=n+1$ (so $\New(f)$ is a simplex), the assumptions that $\New(f)$ is simple at some vertex and that all $\b_j$ lie in the same side of every hyperplane determined by points among $\{\a_1,\ldots,\a_m\}$ clearly hold. In this case, Theorem \ref{npmt-thm1} identifies with Theorem \ref{nc-thm2}. Therefore, Theorem \ref{npmt-thm1} is a generalization of Theorem \ref{nc-thm2} to the case of polynomials with general Newton polytopes and with multiple negative terms.
\end{remark}

\revision{\begin{remark}\label{sec4:rm2}
A polynomial of the form in Theorem \ref{npmt-thm1} for which there exists a point $\bv=[v_k]\in(\R^*)^n$ such that $d_j\bv^{\b_j}>0$ for all $j$ is called {\em orthant-dominated} in \cite{murray}.
\end{remark}}

\begin{example}
Let $f_d=1+x^6+y^6+x^6y^6-x^2y-dx^4y$ and $d^*=\sup\{d\in\R_+\mid f_d\in\PSD\}$. We have $[2,1]^\intercal=\frac{1}{6}[6,6]^\intercal+\frac{1}{6}[6,0]^\intercal+\frac{2}{3}[0,0]^\intercal=\frac{1}{3}[6,0]^\intercal+\frac{1}{6}[0,6]^\intercal+\frac{1}{2}[0,0]^\intercal$, and $[4,1]^\intercal=\frac{1}{6}[6,6]^\intercal+\frac{1}{2}[6,0]^\intercal+\frac{1}{3}[0,0]^\intercal=\frac{2}{3}[6,0]^\intercal+\frac{1}{6}[0,6]^\intercal+\frac{1}{6}[0,0]^\intercal$.
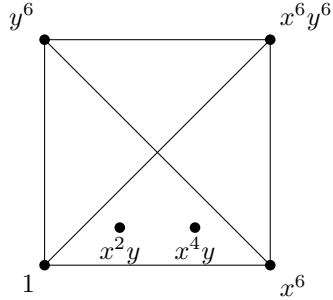
\begin{center}
\begin{tikzpicture}
\draw (0,0)--(0,3);
\draw (0,0)--(3,0);
\draw (3,0)--(3,3);
\draw (0,3)--(3,3);
\draw (0,0)--(3,3);
\draw (0,3)--(3,0);
\fill (0,0) circle (2pt);
\node[below left] (1) at (0,0) {$1$};
\fill (3,0) circle (2pt);
\node[below right] (2) at (3,0) {$x^6$};
\fill (0,3) circle (2pt);
\node[above left] (3) at (0,3) {$y^6$};
\fill (3,3) circle (2pt);
\node[above right] (4) at (3,3) {$x^6y^6$};
\fill (1,1/2) circle (2pt);
\node[below] (5) at (1,1/2) {$x^2y$};
\fill (2,1/2) circle (2pt);
\node[below] (6) at (2,1/2) {$x^4y$};
\node (7) at (0.7,1.2) {$\Delta_2$};
\node (8) at (2.3,1.2) {$\Delta_1$};
\fill[fill=green,fill opacity=0.3] (0,0)--(0,3)--(3,0);
\fill[fill=blue,fill opacity=0.3] (0,0)--(3,0)--(3,3);
\end{tikzpicture}
\end{center}
The system of equations $\{f_d=0,\nabla f_d=\mathbf{0}\}$ in variables $\{x,y,d\}$ has exactly one zero $(x_*\approx1.04521,y_*\approx0.764724,d^*\approx2.11373)$ in $\R_+^{3}$. The linear system \eqref{npmt-eq2} becomes
\begin{equation}\label{npmt-eq12}
\begin{cases}
\frac{2}{3}s_{11}+\frac{1}{2}s_{12}+\frac{1}{3}s_{21}+\frac{1}{6}s_{22}=1\\
\frac{1}{6}s_{11}+\frac{1}{3}s_{12}+\frac{1}{2}s_{21}+\frac{2}{3}s_{22}=x_*^6\\
\frac{1}{6}s_{12}+\frac{1}{6}s_{22}=y_*^6\\
\frac{1}{6}s_{11}+\frac{1}{6}s_{21}=x_*^6y_*^6\\
s_{11}+s_{12}=x_*^2y_*\\
s_{21}+s_{22}=d^*x_*^4y_*
\end{cases}
\end{equation}
which has a nonnegative solution $(s_{11}\approx0.835429,s_{12}=0,s_{21}\approx0.729142,s_{22}=1.2)$. Thus from the proof of Theorem \ref{npmt-thm1}, we obtain a SONC decomposition of $f_{d^*}$ which is $f_{d^*}\approx(0.556953+0.106793x^6+0.533967x^6y^6-x^2y)+(0.243047+0.27962x^6+0.466033x^6y^6-0.798909x^4y)+(0.2+0.613587x^6+y^6-1.31482x^4y)$.
\end{example}

\begin{corollary}\label{npmt-cor1}
Let $f=\sum_{i=1}^mc_i \x^{\a_i}-\sum_{j=1}^ld_j\x^{\b_j}\in\R[\x]$ with $\a_i\in(2\N)^n,c_i\in\R_+,i=1,\ldots,m$, $\b_j\in\New(f)^{\circ},d_j\in\R_+,j=1,\ldots,l$ and $\dim(\New(f))=n$. Assume that $f$ is nonnegative and has a zero, $\New(f)$ is simple at some vertex, and all $\b_j$ lie in the same side of every hyperplane determined by points among $\{\a_1,\ldots,\a_m\}$. Then $f$ has exactly one zero in $\R_+^{n}$.
\end{corollary}
\begin{proof}
By Theorem \ref{npmt-thm1}, $f\in\SONC$. Suppose $f=\sum_{k=1}^rf_k$, where $f_k$ is a nonnegative circuit polynomial for all $k$. Let $\x_{*}$ be a zero of $f$. Then we have $f_k(\x_*)=0$ for all $k$. By Proposition \ref{nc-prop1}, $f_k(|\x_*|)=0$ and $|\x_*|$ is the only zero of $f_k$ in $\R_+^{n}$ for all $k$. Hence $|\x_*|$ is the only zero of $f$ in $\R_+^{n}$.
\end{proof}

%\begin{equation}
%\sum_{i=1}^m c_i(\a_i-\b_l)\mathbf{x}^{\a_i}-\sum_{j=1}^{l-1}d_j(\b_j-\b_l)\mathbf{x}^{\b_j}=0,
%\end{equation}
%where $\a_i\in(2\mathbb{N})^n,c_i>0,i=1,\ldots,m$, $\b_j\in\mathbb{N}^n,d_j>0,j=1,\ldots,l$.

In the remainder of this section, we give an example to illustrate that the condition that all $\b_j$ lie in the same side of every hyperplane determined by points among $\{\a_1,\ldots,\a_m\}$ in Theorem \ref{npmt-thm1} cannot be dropped.
\begin{example}
Let $f=1+4x^2+x^4-3x-3x^3$. Then $f\in\PSD$, but $f\notin\SONC$.
\end{example}
\begin{proof}
It is easy to verify that the minimum of $f$ is $0$ with the only minimizer $x_*=1$. By Theorem \ref{sec3-thm2} (which will be proved in the next section), to get a SONC decomposition for $f$, it suffices to consider the circuits: $\{0,2,1\},\{0,4,1\},\{0,4,3\},\{2,4,3\}$. We have $1=\frac{1}{2}\cdot0+\frac{1}{2}\cdot2=\frac{3}{4}\cdot0+\frac{1}{4}\cdot4$, and $3=\frac{1}{2}\cdot2+\frac{1}{2}\cdot4=\frac{1}{4}\cdot0+\frac{3}{4}\cdot4$.
\begin{center}
\begin{tikzpicture}
\draw (0,0)--(4,0);
\fill (0,0) circle (2pt);
\node[above] (1) at (0,0) {$1$};
\fill (1,0) circle (2pt);
\node[above] (2) at (1,0) {$x$};
\fill (2,0) circle (2pt);
\node[above] (3) at (2,0) {$x^2$};
\fill (3,0) circle (2pt);
\node[above] (4) at (3,0) {$x^3$};
\fill (4,0) circle (2pt);
\node[above] (5) at (4,0) {$x^4$};
\end{tikzpicture}
\end{center}
From the proof of Theorem \ref{npmt-thm1}, we have that if $f\in\SONC$, then the following linear system
\begin{equation}\label{npmt-eq5}
\begin{cases}
\frac{1}{2}s_1+\frac{3}{4}s_3+\frac{1}{4}s_4=1\\
\frac{1}{2}s_1+\frac{1}{2}s_3=4x_*^2\\
\frac{1}{4}s_2+\frac{1}{2}s_3+\frac{3}{4}s_4=x_*^4\\
s_1+s_2=3x_*\\
s_3+s_4=3x_*^3
\end{cases}
\end{equation}
in variables $\{s_1,s_2,s_3,s_4\}$ should have a nonnegative solution. However, (\ref{npmt-eq5}) has no nonnegative solution. This contradictory implies $f\notin\SONC$.
\end{proof}

\section{SONC decompositions preserve sparsity}\label{sec5}
For a nonnegative polynomial $f\in\R[\x]$, let $\Lambda(f):=\{\a\in\supp(f)\mid\a\in(2\N)^n\textrm{ and }c_{\a}>0\}$ (corresponding to the positive terms)
and $\Gamma(f):=\supp(f)\backslash\Lambda(f)$ (corresponding to the negative terms). Then we can write $f=\sum_{\a\in\Lambda(f)}c_{\a}\x^{\a}-\sum_{\b\in\Gamma(f)}d_{\b}\x^{\b}$ with $V(\New(f))\subseteq\Lambda(f)$ (Proposition \ref{nc-prop2}). For every $\b\in\Gamma(f)$, let
\begin{equation}
\mathscr{F}(\b):=\{\Delta\mid\Delta\textrm{ is a simplex, } \b\in\Delta^{\circ}, V(\Delta)\subseteq\Lambda(f)\}.
\end{equation}
Consider the following SONC decomposition for $f$:
\begin{equation}\label{sec5-eq}
f=\sum_{\b\in\Gamma(f)}\sum_{\Delta\in\mathscr{F}(\b)}f_{\b\Delta}+\sum_{\a\in\I}c_{\a}\x^{\a},
\end{equation}
where $f_{\b\Delta}$ is a nonnegative circuit polynomial supported on $V(\Delta)\cup\{\b\}$ for each $\Delta$ and $\I=\{\a\in\Lambda(f)\mid\a\notin\bigcup_{\b\in\Gamma(f)}\bigcup_{\Delta\in\mathscr{F}(\b)}V(\Delta)\}$. If $f$ admits a SONC decomposition of the form \eqref{sec5-eq}, then we say that $f$ decomposes into a {\em sum of nonnegative circuit polynomials without cancellation}.

In Theorem \ref{npgp-thm7} and Theorem \ref{npmt-thm1}, we have seen that nonnegative polynomials satisfying certain conditions decompose into sums of nonnegative circuit polynomials without cancellation. In this section, we shall prove that in fact every SONC polynomial decomposes into a sum of nonnegative circuit polynomials without cancellation. To this end, we first recall a connection between nonnegative circuit polynomials and sums of binomial squares (SBS).

\subsection{Nonnegative circuit polynomials and sums of binomial squares}
% We call a lattice point is {\em even} if it is in $(2\N)^n$.
For a subset $M\subseteq\N^n$, define $\overline{A}(M):=\{\frac{1}{2}(\bu+\bv)\mid\bu\ne\bv,\bu,\bv\in M\cap(2\N)^n\}$ as the set of averages of distinct even lattice points in $M$. For a trellis $\A$, we say that $M$ is an {\em $\A$-mediated set} if $\A\subseteq M\subseteq\overline{A}(M)\cup\A$ \cite{re1}. It turns out that the problem whether a nonnegative circuit polynomial is an SOS polynomial is closely related to $\A$-mediated sets; see Theorem 5.2 in \cite{iw}. The following theorem states that for a nonnegative circuit polynomial $f=\sum_{\a\in\A} c_{\a}\x^{\a}-d\x^{\b}$, if $\b$ belongs to an $\A$-mediated set, then $f$ is actually a sum of binomial squares.
\begin{theorem}\label{sec3-thm1}
Let $f=\sum_{\a\in\A} c_{\a}\x^{\a}-d\x^{\b}\in\R[\x], d\ne0$ be a nonnegative circuit polynomial with $\b\in\New(f)^{\circ}$. If $\b$ belongs to an $\A$-mediated set $M$, then $f$ is a sum of binomial squares, i.e., $f=\sum_{2\bu,2\bv\in M}(a_{\bu}\x^{\bu}-b_{\bv}\x^{\bv})^2$ for some $a_{\bu},b_{\bv}\in\R$.
\end{theorem}
\begin{proof}
The proof can be easily derived from Theorem 5.2 in \cite{iw} and Theorem 4.4 in \cite{re1}.
\end{proof}

Mediated sets were firstly studied by Reznick in \cite{re1}. For a trellis $\A$, there is a maximal $\A$-mediated set $\A^*$ satisfying $\overline{A}(\A)\subseteq\A^*\subseteq\Conv(\A)\cap\N^n$ which contains every $\A$-mediated set. Following \cite{re1}, a trellis $\A$ is called an {\em $H$-trellis} if $\A^*=\Conv(\A)\cap\N^n$. The following theorem states that every trellis is an $H$-trellis after multiplied by a sufficiently large integer.
\begin{theorem}(\cite[Theorem 3.5]{po})\label{sec3-prop}
Let $\A\subseteq\N^n$ be a trellis. Then $k\A$ is an $H$-trellis for any integer $k\ge n$.
\end{theorem}
\begin{remark}
The polynomials in \cite{po} were assumed to be homogeneous. So we need $k\ge n$ instead of $k\ge n-1$ to adapt to our situation.
\end{remark}

From Theorem \ref{sec3-prop} together with Theorem \ref{sec3-thm1}, we know that every $n$-variate nonnegative circuit polynomial supported on $k\A$ and a lattice point in the interior of $\Conv(k\A)$ is a sum of binomial squares for any trellis $\A$ and an integer $k\ge n$.
\begin{lemma}\label{sec4-lm1}
Suppose that $f(x_1,\ldots,x_n)\in\R[\x]$ is a SONC polynomial. Then $f(x_1^k,\ldots,x_n^k)$ is a sum of binomial squares for any integer $k\ge n$.
\end{lemma}
\begin{proof}
Assume $f=\sum_i f_i$, where all $f_i$ are nonnegative circuit polynomials. For any integer $k\ge n$, since every $f_i(x_1^k,\ldots,x_n^k)$ is a sum of binomial squares (by Theorem \ref{sec3-thm1} and Theorem \ref{sec3-prop}), so is $f(x_1^k,\ldots,x_n^k)$.
\end{proof}

\subsection{SONC decompostions without cancellation}\label{sonc-wc}
Now we can prove: every SONC polynomial decomposes into a sum of nonnegative circuit polynomials without cancellation. The proof takes use of the SBS decompositions for SONC polynomials. The following lemma enables us to consider $f(x_1^k,\ldots,x_n^k)$ instead of $f(x_1,\ldots,x_n)$ for an odd number $k$.
\begin{lemma}\label{sec4-lm}
Let $f(x_1,\ldots,x_n)\in\R[\x]$ and $k\in\N$ be an odd number. Then $f(x_1,\ldots,x_n)$ decomposes into a sum of nonnegative circuit polynomials without cancellation if and only if $f(x_1^k,\ldots,x_n^k)$ decomposes into a sum of nonnegative circuit polynomials without cancellation.
\end{lemma}
\begin{proof}
Notice that for an odd number $k$, $f(x_1,\ldots,x_n)$ is a nonnegative circuit polynomial if and only if $f(x_1^k,\ldots,x_n^k)$ is a nonnegative circuit polynomial. The lemma then follows from this fact easily.
\end{proof}

% If a nonnegative polynomial $f\in\R[\x]$ has the form $\sum_{\a\in\Lambda(f)}c_{\a}\x^{\a}-d\x^{\b}$, where either $\b\in(2\N)^n$ and $d>0$, or $\b\notin(2\N)^n$, then we call $f$ an {\em AGE polynomial}. 
If a nonnegative polynomial $f$ can be written as
\revision{\begin{equation}\label{sage}
f=\sum_{\b\in\Gamma(f)}(\sum_{\a\in\Lambda(f)}c_{\b\a}\x^{\a}-d_{\b}\x^{\b})
\end{equation}}
such that every $\sum_{\a\in\Lambda(f)}c_{\b\a}\x^{\a}-d_{\b}\x^{\b}$ is an AGE polynomial, then we say that $f$ decomposes into a {\em sum of AGE polynomials without cancellation}. \revision{By Theorem \ref{npgp-thm8}, every AGE polynomial decomposes into a sum of nonnegative circuit polynomials without cancellation. So if $f$ decomposes into a sum of AGE polynomials without cancellation, then $f$ also decomposes into a sum of nonnegative circuit polynomials without cancellation.}
\begin{theorem}\label{sec3-thm2}
Let $f=\sum_{\a\in\Lambda(f)}c_{\a}\x^{\a}-\sum_{\b\in\Gamma(f)}d_{\b}\x^{\b}\in\R[\x]$. If $f\in\SONC$, then $f$ decomposes into a sum of nonnegative circuit polynomials without cancellation, i.e., $f$ admits a SONC decomposition of the form \eqref{sec5-eq}.
\end{theorem}
\begin{proof}
By Lemma \ref{sec4-lm}, we only need to prove the theorem for $f(x_1^{2n+1},\ldots,x_n^{2n+1})$. In view of Theorem \ref{npgp-thm8}, we complete the proof by showing that $f(x_1^{2n+1},\ldots,x_n^{2n+1})$ decomposes into a sum of AGE polynomials without cancellation.

For simplicity, let $h=f(x_1^{2n+1},\ldots,x_n^{2n+1})$. By Lemma \ref{sec4-lm1}, we can write $h$ as a sum of binomial squares, i.e. $h=\sum_{i=1}^m(a_i\x^{\bu_i}-b_i\x^{\bv_i})^2$. To prove that $h$ decomposes into a sum of AGE polynomials without cancellation, let us do induction on $m$. When $m=1$, $h=(a_1\x^{\bu_1}-b_1\x^{\bv_1})^2=a_1^2\x^{2\bu_1}+b_1^2\x^{2\bv_1}-2a_1b_1\x^{\bu_1+\bv_1}$ and the conclusion obviously holds. Assume that the conclusion is correct for $m-1$ and now consider the case of $m$.
% Without loss of generality, assume $\bu_m+\bv_m\in\Gamma(h)$. 
Let $h'=\sum_{i=1}^{m-1}(a_i\x^{\bu_i}-b_i\x^{\bv_i})^2=\sum_{\a\in\Lambda(h')}c_{\a}'\x^{\a}-\sum_{\b\in\Gamma(h')}d_{\b}'\x^{\b}$. By the induction hypothesis, we can write $h'=\sum_{\b\in\Gamma(h')}(\sum_{\a\in\Lambda(h')}c_{\b\a}'\x^{\a}-d_{\b}'\x^{\b})$ as a sum of AGE polynomials without cancellation. Then,
\begin{equation}\label{sec3-eq10}
h=\sum_{\b\in\Gamma(h')}(\sum_{\a\in\Lambda(h')}c_{\b\a}'\x^{\a}-d_{\b}'\x^{\b})+(a_m\x^{\bu_m}-b_m\x^{\bv_m})^2.
\end{equation}
From $h=h'+(a_m\x^{\bu_m}-b_m\x^{\bv_m})^2$, \revision{it follows that cancellations in \eqref{sec3-eq10} only occur among terms involving $\x^{2\bu_m}$, $\x^{2\bv_m}$, $\x^{\bu_m+\bv_m}$.
Our goal is to rewrite $h$ as a sum of AGE polynomials without cancellation by adjusting the terms involving $\x^{2\bu_m}$, $\x^{2\bv_m}$, $\x^{\bu_m+\bv_m}$ in (\ref{sec3-eq10})}.

First let us consider the terms involving $\x^{2\bu_m}$. If $2\bu_m\notin\Gamma(h')$, then we have nothing to do. If $2\bu_m\in\Gamma(h')$ and $2\bu_m\in\Gamma(h)$, then we must have $d_{2\bu_m}'>a_m^2$. By the equality
\begin{align*}
    &(\sum_{\a\in\Lambda(h')}c_{2\bu_m\a}'\x^{\a}-d_{2\bu_m}'\x^{2\bu_m})+a_m^2\x^{2\bu_m}+b_m^2\x^{2\bv_m}-2a_mb_m\x^{\bu_m+\bv_m}=\\
    &(1-\frac{a_m^2}{d_{2\bu_m}'})(\sum_{\a\in\Lambda(h')}c_{2\bu_m\a}'\x^{\a}
-d_{2\bu_m}'\x^{2\bu_m})+\sum_{\a\in\Lambda(h')}\frac{c_{2\bu_m\a}'a_m^2}{d_{2\bu_m}'}\x^{\a}+b_m^2\x^{2\bv_m}-2a_mb_m\x^{\bu_m+\bv_m},
\end{align*}
we obtain 
\begin{align*}
h=&\sum_{\b\in\Gamma(h')\backslash\{2\bu_m\}}(\sum_{\a\in\Lambda(h')}c_{\b\a}'\x^{\a}-d_{\b}'\x^{\b})+(1-\frac{a_m^2}{d_{2\bu_m}'})(\sum_{\a\in\Lambda(h')}c_{2\bu_m\a}'\x^{\a}
-d_{2\bu_m}'\x^{2\bu_m})\\
&+(\sum_{\a\in\Lambda(h')}\frac{c_{2\bu_m\a}'a_m^2}{d_{2\bu_m}'}\x^{\a}+b_m^2\x^{2\bv_m}-2a_mb_m\x^{\bu_m+\bv_m}),
\end{align*}
\revision{which is a sum of AGE polynomials without cancellation among terms involving $\x^{2\bu_m}$}. If $2\bu_m\in\Gamma(h')$ and $2\bu_m\in\Lambda(h)$, then we must have $a_m^2>d_{2\bu_m}'$, and we can write $h$ as
\begin{align*}
h=&\sum_{\b\in\Gamma(h')\backslash\{2\bu_m\}}(\sum_{\a\in\Lambda(h')}c_{\b\a}'\x^{\a}-d_{\b}'\x^{\b})\\
&+(\sum_{\a\in\Lambda(h')}c_{2\bu_m\a}'\x^{\a}
+(a_m^2-d_{2\bu_m}')\x^{2\bu_m}+b_m^2\x^{2\bv_m}-2a_mb_m\x^{\bu_m+\bv_m}),
\end{align*}
\revision{which is a sum of AGE polynomials without cancellation among terms involving $\x^{2\bu_m}$}. If $2\bu_m\in\Gamma(h')$ and $2\bu_m\notin\supp(h)$, then the terms $-d_{2\bu_m}'\x^{2\bu_m}$ and $a_m^2\x^{2\bu_m}$ must be cancelled in (\ref{sec3-eq10}). Hence we obtain the expression of $h$ as
\begin{align*}
h=\sum_{\b\in\Gamma(h')\backslash\{2\bu_m\}}(\sum_{\a\in\Lambda(h')}c_{\b\a}'\x^{\a}-d_{\b}'\x^{\b})+(\sum_{\a\in\Lambda(h')}c_{2\bu_m\a}'\x^{\a}+b_m^2\x^{2\bv_m}-2a_mb_m\x^{\bu_m+\bv_m}),
\end{align*}
\revision{which is a sum of AGE polynomials without cancellation among terms involving $\x^{2\bu_m}$}.

Continue to adjust the terms involving $\x^{2\bv_m}$ and $\x^{\bu_m+\bv_m}$ in the expression of $h$ in a similar way. Finally we can write $h$ as a sum of AGE polynomials without cancellation as desired.
\end{proof}

Due to the coincidence of the SONC cone and the SAGE cone, it is immediate from Theorem \ref{sec3-thm2} that every SAGE polynomial decomposes into a sum of AGE polynomials without cancellation. \revision{This result was also independently proved in \cite[Corollary 20]{murray} by different techniques. It was first proved in the context of signomials based on convex duality and then was specialized to the situation of polynomials. The proof given here, however, deals directly with polynomials and employs integrality of exponents in an essential way.} 

Theorem \ref{sec3-thm2} ensures that every SONC polynomial admits a SONC decomposition by using only the support from the original polynomial and with no cancellation. This is a very desired property (sparsity-preservation) to design efficient algorithms for sparse polynomial optimization based on SONC decompositions and is a distinguished difference from SOS decompositions. In the SOS case, a well-known result concerning sparsity due to Reznick states that if $f=\sum_if_i^2$, then $\supp(f_i)\subseteq\frac{1}{2}\New(f)$ (\cite[Theorem 1]{reznick1978extremal}), but generally cancellations occur among $f_i^2$'s. As a simple example, consider $f=3-4x+x^4=2(1-x)^2+(1-x^2)^2$. When expanding the squares on the right side, the monomial $x^2$ appears though it doesn't appear in the expression of $f$.

When applying SONC certificates to unconstrained polynomial optimization (i.e., minimizing a polynomial function over $\R^n$), the first problem is to decide which circuits are needed in construction of SONC decompositions. Once the set of candidate circuits is given, the rest of the computation can be done via relative entropy programming \cite{iw} or second order cone programming \cite{socp}. Hence the overall complexity greatly depends on the number of candidate circuits. Because of Theorem \ref{sec3-thm2}, one may only consider the circuits contained in the support of the input polynomial without loss of generality to avoid enumerating all possible circuits which might be an astronomical number.
In this sense, Theorem \ref{sec3-thm2} is crucial to decrease the number of candidate circuits. However, in general the number of circuits contained in the support of the input polynomial scales combinatorically with the number of terms. \revision{In \cite{katthan}, the notion of reduced circuits was proposed to remove redundant circuits further. It was proved that it suffices to consider reduced circuits to compute SONC decompositions, though the number of reduced circuits still scales combinatorically with the number of terms.}
On the other hand, by Carath\'eodory's theorem (\cite[Corollary 17.1.2]{ro}), it is possible to write a SONC polynomial $f$ as a sum of at most $|\supp(f)|$ nonnegative circuit polynomials. \revision{We still do not know whether there are theoretical obstacles to stop us from obtaining such a SONC decomposition efficiently}. In any case, more efforts are required to further reduce the number of candidate circuits and to make the computation more tractable.

As a comparison, when applying SAGE certificates to unconstrained polynomial optimization, the number of AGE polynomials in construction of SAGE decompositions equals the number of negative terms of the input polynomial owing to Theorem \ref{sec3-thm2}, and deciding whether a polynomial is an AGE polynomial can be performed via relative entropy programming \cite{murray}. Thus the whole computation can be done efficiently for sparse polynomials.

\section{Conclusions and discussions}\label{sec6}
This paper has studied several problems concerning SONC decompositions for nonnegative polynomials. We have proved that nonnegative polynomials with one negative term are SONC polynomials. This result implies that the SONC cone actually coincides with the SAGE cone. Under certain conditions, we have also proved that nonnegative polynomials with multiple negative terms are SONC polynomials. Moreover, we have proved that every SONC polynomial admits a SONC decomposition without cancellation. \revision{Following this line of research}, there are still many questions left for further investigation:

\begin{itemize}
    \item In Theorem \ref{npmt-thm1}, we have used a technical condition that the Newton polytope is simple at some vertex to finish the proof. It is not clear whether this condition can be dropped. The answer seems closely related to the existence of positive zeros for a particular system of polynomial equations (\cite{wang5}).
    \item Even though the number of candidate circuits is significantly reduced thanks to the SONC decomposition without cancellation \eqref{sec5-eq} we have provided, the computation of such a SONC decomposition is still generally intractable because the number of circuits used in \eqref{sec5-eq} scales combinatorically with the number of terms of the input polynomial. As to unconstrained polynomial optimization, one may rely on certain heuristics to obtain a reasonable number of circuits as \cite{se} or \cite{socp} did at the cost of losing some accuracy. One would also like to seek an approach to decrease the number of circuits without losing accuracy. The recent work in \cite{pap} made the first step towards this direction. See also the discussion at the end of Section \ref{sonc-wc}.
    \item The fact that every SONC polynomial after an appropriate dilation of the support is a sum of binomial squares, which we rely on to prove Theorem \ref{sec3-thm2}, indicates the possibility of computing SONC decompositions via second order cone programming. The recent work in \cite{socp} is a good start on this topic.
    \item Another interesting and also important question is to what extent the results of this paper can be generalized to the case of nonnegativity over a subset of $\R^n$? Does the sparsity-preserving property still hold for certain classes of subsets? The answers to these questions would help to leverage SONC certificates to solve constrained polynomial optimization problems. The recent work in \cite{murray3,naumann2021sublinear} can be viewed as attempts towards this direction. 
\end{itemize}

\section*{Acknowledgments}
The author would like to thank the anonymous referees for their helpful suggestions, which have led to a much-improved paper.

\bibliographystyle{siamplain}
\bibliography{refer}
\end{document}